\documentclass[oneside, 12pt]{amsart}

\usepackage{amsmath, amsfonts, amssymb, amsthm, graphics}

\newtheorem{theorem}{Theorem}[section]
\newtheorem{lemma}[theorem]{Lemma}

\theoremstyle{definition}

\DeclareMathOperator{\arccot}{arccot}

\begin{document}

\title[Best constants in reverse inequalities for analytic and co-analytic projections]
{Best constants in reverse Riesz-type inequalities for analytic and co-analytic projections}

\author{Petar  Melentijevi\'{c}}
\address{
Faculty of  Mathematics\endgraf
University of Belgrade\endgraf
Studentski trg 16\endgraf
11000 Beograd\endgraf
Serbia\endgraf}
\email{petarmel@matf.bg.ac.rs}

\begin{abstract}
Let $P_+$ be the Riesz's projection operator and let  $P_-= I - P_+$. We consider the inequalities of the following form
$$ \|f\|_{L^p(\mathbb{T})}\leq B_{p,s}\|( |P_ +  f | ^s + |P_-  f |^s)
^{\frac 1s}\|_{L^p (\mathbb{T})} $$
and prove them with sharp constant $B_{p,s}$ for $s \in (0,1]\cup [p',+\infty)$ and $1<p\leq 2$ or $p\geq 4,$ where $p':=\min\{p,\frac{p}{p-1}\}.$ 
\end{abstract}



\thanks{The author is partially supported by MPNTR Serbia grant no. 174017}
\maketitle

\section{Introduction and main results}

Let $\mathbb{T}=\{z\in \mathbb{C}:|z|=1\}$  be the unit circle in the complex plane. For  $0<p\le\infty,$ we denote by $L^p(\mathbb{T})$ the  Lebesgue space of $p-$integrable functions on $\mathbb{T}$ and by $H^p(\mathbb{T})$ its subspace that contains all $\varphi  \in  L^p(\mathbb{T})$   for which all negative Fourier coefficients are equal to zero, i.e.,
\begin{equation*}
\hat{\varphi}(n)  =  \frac{1}{2\pi}\int_{0}^{2\pi} \varphi (e^{it})e^{-int} dt = 0\quad \text{for every  integer}\  n<0.
\end{equation*}
The Riesz projection  operator $P_+:L^p(\mathbb{T})\to H^p(\mathbb{T})$ is defined as:
\begin{equation*}
P_+f (\zeta)  =  \sum_{n=0}^{+\infty}      \hat{f}(n)  \zeta^n,
\end{equation*}
for $f(\zeta)= \sum_{n =   -\infty}^{+\infty} \hat{f}(n) \zeta^n.$ Its complement to the identity operator, i.e. $P_-=I-P_+$ is called co-analytic projection.  

Let us denote the unit disk by $\mathbb{D}=\{z\in \mathbb{C}:|z|<1\}.$ For a function  $f$ on $\mathbb{D}$ and $r\in (0,1)$ we consider dilates $f_r$ defines as  $f_r (\zeta) =  f(r\zeta)$, $\zeta\in\overline{\mathbb{D}}$. The
harmonic Hardy space $h^p$, $0<p\le \infty$, consists of all harmonic complex-valued functions  $\phi$  on $\mathbb{D}$ for which the integral mean
\begin{equation*}
M_p(f,r)  = \left\{ \int_\mathbb{T} |f_r(\zeta)|^p\frac{|d\zeta|}{2\pi}\right\}^{\frac 1p}
\end{equation*}
remains bounded as $r$ approaches $1$.  Since $|f|^p$, $1\le p<\infty$ is  subharmonic on  $\mathbf{U}$, the integral mean $M_p(f,r)$ is increasing in
$r$. The norm  on $h^p$ in this case is given by
\begin{equation*}
\|f\|_p = \lim_{r\rightarrow 1} M_p(f,r).
\end{equation*}
The analytic Hardy space $H^p$ is the subspace of $h^p$ that contains analytic functions. For the  theory of  Hardy spaces in the unit disk we refer to
classical books \cite{DUREN.BOOK, GARNETT.BOOK, RUDIN.BOOK} and  Pavlovi\'{c} book \cite{PAVLOVIC.BOOK} for more recent results. For holomorphic $f$,  $|f|^p$ is subharmonic for every $0<p<\infty$,  
\begin{equation*}
\|f\|_p  = \lim_{ r \rightarrow 1}  M(r,f)
\end{equation*}
defines a norm on $H^p$ for every $0<p<+\infty.$
For $p=\infty$ the space $H^\infty$ contains all bounded analytic functions in  the unit disc.

For $f\in H^p$ the radial boundary value  $f^{\ast}(\zeta)= \lim_{r\rightarrow 1-} f(r\zeta)$ exists for almost every
$\zeta\in \mathbb{T}$,  $f^\ast$ belongs to the space  $H^p(\mathbb{T})$, and  we have the isometry relation $\|f^{\ast}\|_{L^p(\mathbb{T})} =
\|f\|_{H^p}$ by Fatou's theorem.  Moreover, given $\varphi \in H^p (\mathbb{T})$, the Cauchy (and the  Poisson)  extension  of it  gives  a function  $f$ in  $H^p$ such
that   $f^\ast (\zeta ) = \varphi(\zeta)$ for almost every $\zeta\in \mathbb{T}$. Therefore, one may identify the spaces $H^p(\mathbb{T})$  and  $H^p$
via the   isometry $f\rightarrow f^\ast$.  If we identify  the Hardy space $H^p$ and the space  $H^p(\mathbb{T})$, then the Riesz operator  $P_+$
may be  represented as the  Cauchy integral (with density $f^\ast$)
\begin{equation*}
P_+ f (z) =  \frac{1}{2\pi i} \int_{\mathbb{T}} \frac{f^\ast(\zeta)}{\zeta - z} d\zeta, \quad z \in \mathbb{D}.
\end{equation*}

 It is a classical result proved by Marsel Riesz that  $P_+$ is a bounded operator on $L^p(\mathbb{T})$ for every $1<p<\infty$. The problem of finding the exact norm is more demanding, and after some partial results by Gohberg and Krupnik (\cite{GOHBERGKRUPNIK}) and Verbitsky (\cite{VERBITSKY.ISSLED}), was finally solved in 2000 by Hollenbeck and Verbitsky (\cite{HV.JFA}). They used the method of plurisubharmonic minorants and proved slightly stronger inequality 
$$\|  \max\{ |P_ +  f | ,|P_-  f |\}   \|_{L^p (\mathbb{T})}\le \frac{1}{\sin\frac{\pi}{p}}\|f\|_{L^p (\mathbb{T})},
\quad f\in {L^p (\mathbb{T})}, \,    1<p<2.$$
This proves result for $1<p<2,$ while the case $p>2$ follows by duality argument. Later on, the same authors proved the analogous inequality for $p>2$ in \cite{HV.OTAA}. In the same paper, they   
posed a problem to find the optimal constant $A _{p,s }$  for the inequality
\begin{equation*}
\|  ( |P_ +  f | ^s  +  |P_-  f |^s) ^{\frac 1s}  \|_{L^p (\mathbb{T})}\le A_{p,s } \|f\|_{L^p (\mathbb{T})},
\quad f\in {L^p (\mathbb{T})}, \,    1<p<\infty, \, 0<s<\infty.
\end{equation*}
Note that Hollenbeck-Verbitsky result from \cite{HV.JFA} gives $A_{p, \infty }  = \frac {1} {\sin  \frac  \pi p}.$\\
 In \cite{KALAJ.TAMS} Kalaj confirmed this conjecture for $s=2.$ The author, in his thesis \cite{MELENTIJEVIC.PHD}, had realized that inequalities in \cite{KALAJ.TAMS} can be proved using a different method. It has turned out that this approach can be extended and in \cite{PETARMARIJAN}, the author with Markovi\'{c} proved these inequalities for $s\leq\max\{\frac{p}{p-1}, p\}$ and $1<p\leq\frac{5}{4}$ or $p\geq 2.$ Recently, in \cite{PETAR}, this conjecture is proved in full generality for $p\geq 2$ and $s\leq\frac{1}{\cos^2\frac{\pi}{2p}}$ for $1<p\leq \frac{4}{3}.$ It is also proved that in the second case the result cannot achieved in the same way. Very recently, in \cite{JAGUZOVIC}, this conjecture is verified for $\frac{4}{3}\leq p\leq 2$ and $s\leq \frac{p}{p-1}.$ 
By Banach homeomorphism theorem, there are the constants $B_{p,s}$ such that 
$$\|f\|_{L^p (\mathbb{T})}\leq B_{p,s } \|  ( |P_ +  f | ^s  +  |P_-  f |^s) ^{\frac 1s}  \|_{L^p (\mathbb{T})},
\quad f\in {L^p (\mathbb{T})}, \,    1<p<\infty, \, 0<s<\infty.$$
In this paper, we will estimate $B_{p,s}$ from below on the whole range and find its exact value for $s \in (0,1]$ and every $1<p<+\infty$ and $s\geq\min\{p,\frac{p}{p-1}\}$ and $p \in (1,2] \cup [4,+\infty).$ 
More precisely, the main result of this paper is given by the following theorem.
\begin{theorem}
1) For $1<p\leq 2$ and $s\geq p,$ there holds the sharp inequality:
$$\|f\|_{L^p (\mathbb{T})}\leq 2^{1-\frac{1}{s}}\sin\frac{\pi}{2p} \|  ( |P_ +  f | ^s  +  |P_-  f |^s) ^{\frac 1s}  \|_{L^p (\mathbb{T})},
\quad f\in {L^p (\mathbb{T})},$$
2) If $p\geq 4$ and $s\geq \frac{p}{p-1},$ we have:
$$\|f\|_{L^p (\mathbb{T})}\leq 2^{1-\frac{1}{s}}\cos\frac{\pi}{2p} \|  ( |P_ +  f | ^s  +  |P_-  f |^s) ^{\frac 1s}  \|_{L^p (\mathbb{T})},
\quad f\in {L^p (\mathbb{T})},$$
3) For $s \in (0,1]$ and any $1<p<+\infty,$ we have:
$$\|f\|_{L^p (\mathbb{T})}\leq \|  ( |P_ +  f | ^s  +  |P_-  f |^s) ^{\frac 1s}  \|_{L^p (\mathbb{T})},
\quad f\in {L^p (\mathbb{T})}.$$
\end{theorem}
It is easy to see that 3) is trivial, an explanation will be given in the next section. For 1) and 2), we note that we extends the results from \cite{KALAJ.TAMS}, where these inequalities are proved for $s=2.$ Similar inequalities, with additional parameter and for $1<p<2$ are recently proved in \cite{KALAJ.POTA}.\\

Interested reader can easily formulate the appropriate analogues for half-line or half-space multipliers or analytic martingales following the corresponding results from \cite{PETAR}.\\

In section 2 we will give estimates for $B_{p,s}$ using well-known family of nearly-extremal functions. Section 3 contains proof of our theorem modulo ''elementary'' inequalities for complex numbers. Further reduction of these inequalities is given in section 4, while in sections 5 and 6 we give proofs of these inequalities in the ''hardest'' case - that of $t=0.$

\section{Lower estimates of the constants -- the family of test functions}

In this section we give the estimate for the constants $B _{p,s}$ in the inequalities

\begin{equation*}
\|f\|_{L^p(\mathbb{T})} \le B _{p,s } \|\left(|P_+f|^s+|P_-f|^s \right)^{1/s}\|_{L^p(\mathbb{T})}
\end{equation*}
from below.

We will use the family of test functions defined by
\begin{equation*}
f_{\gamma}= \alpha \Re g_{\gamma} + i \beta \Im g_{\gamma},
\end{equation*}
where $g_{\gamma}(z)= \left(\frac{1+z}{1-z}\right)^{\frac{2\gamma}{\pi}}$. Note that  $|\Re g_{\gamma}|=
\tan \gamma |\Im g_{\gamma}|.$  Hence, following the calculation in \cite{HV.OTAA} or \cite{PETAR}, for  $\gamma$ tends to $\frac{\pi}{p},$ we have:
\begin{equation*}
\frac{\|\left(|P_+f|^s+|P_-f|^s \right)^{\frac{1}{s}}\|_{L^p(\mathbb{T})}}{\|f\|_{L^p(\mathbb{T})}}=
\frac{\left(|\alpha+\beta|^s+|\alpha-\beta|^s\right)^{\frac{1}{s}}}{2\left(\alpha^2 \cos^2 \frac{\pi}{2p}
	+ \beta^2 \sin^2 \frac{\pi}{2p}\right)^{\frac{1}{2}}}.
\end{equation*}

Therefore, we easily conclude that
\begin{equation*}\begin{split}
B_{p,s}  & \ge \sup_{\alpha,\beta \in \mathbb{R}} \frac{2\big(\alpha^2 \cos^2 \frac{\pi}{2p} + \beta^2 \sin^2 \frac{\pi}{2p}\big)^{\frac{1}{2}}}{\big(|\alpha+\beta|^s+|\alpha-\beta|^s\big)^{\frac{1}{s}}}
\\ & =  \bigg(\inf_{\alpha,\beta \in \mathbb{R}}
\frac{\big(|\alpha+\beta|^s+|\alpha-\beta|^s\big)^{\frac{1}{s}}}
{2\big(\alpha^2 \cos^2 \frac{\pi}{2p} + \beta^2 \sin^2\frac{\pi}{2p}\big)^{\frac{1}{2}}}\bigg)^{-1}.
\end{split}\end{equation*}
Naturally, we analyze the function
\begin{equation*}
T(\alpha,\beta)= \frac{\big(|\alpha+\beta|^s+|\alpha-\beta|^s\big)^{\frac{1}{s}}}{2\big(\alpha^2 \cos^2 \frac{\pi}{2p} + \beta^2 \sin^2 \frac{\pi}{2p}\big)^{\frac{1}{2}}}.
\end{equation*}
This function is even in both variables, and we can assume, without loss of generality, that $|\alpha+\beta| \ge  |\alpha-\beta|$. Also,
it is homogeneous, and we can set $\alpha+\beta=1$ and $\alpha-\beta=t,$ with $|t|\leq 1,$ by the  previous assumption. Therefore, we get:
\begin{equation*}\begin{split}
\frac{\big(|\alpha+\beta|^s+|\alpha-\beta|^s\big)^{\frac{1}{s}}}{2\big(\alpha^2 \cos^2 \frac{\pi}{2p} + \beta^2 \sin^2
	\frac{\pi}{2p}\big)^{\frac{1}{2}}}&
=\frac{\big(1+|t|^s\big)^{\frac{1}{s}}}{2\big ((\frac{1+t}{2})^2 \cos^2 \frac{\pi}{2p} + (\frac{1-t}{2})^2 \sin^2 \frac{\pi}{2p}\big)^{\frac{1}{2}}}
\\&=\frac{\big(1+|t|^s\big)^{\frac{1}{s}}}{\sqrt{1+t^2+2t\cos \frac{\pi}{p}}}.
\end{split}\end{equation*}
Hence, the range of the function $T(\alpha,\beta)$ is the union of the values of the expressions:
\begin{equation*}
\frac{\big(1+t^s\big)^{\frac{1}{s}}}{\sqrt{1+t^2+2t\cos \frac{\pi}{p}}}
\end{equation*}
and
\begin{equation*}
\frac{\big(1+t^s\big)^{\frac{1}{s}}}{\sqrt{1+t^2-2t\cos \frac{\pi}{p}}}
\end{equation*}
for $0 \leq   t \leq  1$.

Since, for $1<p\leq 2$:
\begin{equation*}
\sqrt{1+t^2-2t\cos \frac{\pi}{p}} \geq \sqrt{1+t^2+2t\cos \frac{\pi}{p}},
\end{equation*}
we conclude:
\begin{equation*}
B_{p,s} \geq \left(\inf_{t \in [0,1]} \frac{\big(1+t^s\big)^{\frac{1}{s}}}{\sqrt{1+t^2-2t\cos \frac{\pi}{p}}}\right)^{-1}.
\end{equation*}

 We consider the function $G(t)=\frac{(1+t^s)^2}{(1+t^2-2t\cos\frac{\pi}{p})^s}$ for $0\leq t \leq 1.$
A straighforward calculation gives$G'(t)=\frac{2s(1+t^s)}{(1+t^2-2t\cos\frac{\pi}{p})^{s+1}}\big[t^{s-1}-t+\cos\frac{\pi}{p}(1-t^s)\big]$ and we have
$\mathrm {sgn} G'(t)   = \mathrm {sgn}\Psi(t),$
where $\Psi(t)= t^{s-1}-t+\cos\frac{\pi}{p}(1-t^s).$ Now, from $\Psi'(t)= (s-1)t^{s-2}-1-st^{s-1}\cos\frac{\pi}{p}$ and $\Psi''(t)=(s-1)t^{s-3}\big[s-2-ts\cos\frac{\pi}{p}\big]$ we make difference between these three cases:

1) For $\frac{1}{\cos^2\frac{\pi}{2p}}\leq s<2,$ we have $s-2-s\cos\frac{\pi}{p}>0$ and also $s>1,$ so $s-2-ts\cos\frac{\pi}{p}$ increases on $t$
and, since $s-2<0$ and $s-2-s\cos\frac{\pi}{p},$ we see that for $t_0=\frac{s-2}{s\cos\frac{\pi}{p}}$ we have $\Psi''(t)<0$ for $t\in
(0,t_0)$ and $\Psi''(t)>0$ for $t\in (t_0,1).$ Hence, $\Psi'$ decreases on
$(0,t_0)$ and increases on $(t_0,1),$ hence $\Psi'(t_0)$ is the minimum of $\Psi'$ on $(0,1).$ But, $\Psi'(t_0)=(s-1)t_0^{s-2}-1-st_0^{s-1}\cos\frac{\pi}{p}=(s-1)t_0^{s-2}-1-(s-2)t_0^{s-2}=t_0^{s-2}-1>0,$ since $s<2$ and $t_0<1$.  This means
that $\Psi'(t)\geq\Psi'(t_0)>0,$ so $\Psi$ increases and $\Psi(t)\leq \Psi(1)=0,$ which implies that $G$ has its minimum for $t=1$ equal to $G(1)=\frac{4^{1-s}}{\sin^{2s}\frac{\pi}{2p}}$ and $B_{p,s}\geq 2^{1-\frac{1}{s}} \sin\frac{\pi}{2p}.$

2) For $s\geq 2,$ we infer that $\Psi''(t)>0 $ for $t \in (0,1),$ function $\Psi$ has its maximum in $t=1$ equal to $0,$ therefore: $B_{p,s}\geq 2^{1-\frac{1}{s}} \sin\frac{\pi}{2p}.$

3) For $1<s<\frac{1}{\cos^2\frac{\pi}{2p}}$the expression  $s-2-s\cos\frac{\pi}{p}$ is negative which gives $\Psi''(t)<0$ and $\Psi'$ decreases. $\Psi'(1)=s-2-s\cos\frac{\pi}{p}<0,$ while $\lim_{t\rightarrow 0+}\Psi'(t)= +\infty$ so there exists a unique $t_0$ such that $\Psi'(t_0)=0$.
Hence, $\Psi(t)$ increases on $(0,t_0)$ and decreases on $(t_0,1).$ Since $\Psi(0)=\cos\frac{\pi}{p}<0,$ $\Psi(1)=0$ we conclude $\Psi(t_0)>0$
and $\Psi $ has exactly one zero $\tilde{t}$ in which $G(t)$ attains its minimum. Therefore, for
such the point $\tilde{t},$ we have  $B_{p,s} \geq \frac{\sqrt{1+\tilde{t}^2-2\tilde{t}\cos\frac{\pi}{p}}}{(1+\tilde{t}^s)^{\frac{1}{s}}}$.

4)	For $0<s\leq1$ we have $\Psi''(t)>0$ for every $t \in [0,1],$ so $\Psi'$ increases, $\Psi'(t) \leq \Psi'(1)=0$ and $\Psi$ decreases. But,
$\Psi(1)=0$ so $\Psi(t)\geq \Psi(1)=0,$ $G$ is increasing and has minimum in zero. In this case, we conclude $B_{p,s} \geq  1.$ 	In the other hand, since $t\mapsto t^{s}$ is concave function, we have:
$$\|f\|_{L^p(\mathbb{T})}=\|P_+f+P_-f\|_{L^p(\mathbb{T})}\leq \||P_+f|+|P_-f|\|_{L^p(\mathbb{T})}$$$$\leq(|P_+f|^s+|P_-f|^s)^{\frac{1}{s}}\|_{L^p(\mathbb{T})},$$
which gives $B_{p,s}=1.$

The rest easily follows from the simple observation: If $p>2,$ than $\cos\frac{\pi}{p}=-\cos\frac{\pi}{p'}$ for dual exponent $p'=\frac{p}{p-1}<2.$ Consequently, we can replace the appropriate signs and functions sine and cosine.

Case $s\geq 2,$ already follows from Kalaj's paper (\cite{KALAJ.TAMS})(where it was proved for $s=2$) and power-mean inequality, hence, we have $B_{p,s}=2^{1-\frac{1}{s}}\sin\frac{\pi}{2p}$ in that case. However, we will prove stronger result in this paper (with smaller $s$).

\section{Proof of Theorem 1.1}

In this section we will give proof of Theorem 1.1. First, we will formulate two lemmata whose proofs will be given through remaining sections.  

\begin{lemma}
For $s\geq p$ and $1<p\leq 2,$ there holds
\begin{equation}
\label{lemma31}
-\left(\frac{|z|^s+|w|^s}{2}\right)^{\frac ps}+\frac{|z+\overline{w}|^p}{2^p\sin^p \frac{\pi}{2p} } +\cot\frac{\pi}{2p}\Phi_1(z,w)
\leq 0,
\end{equation}
where
$\Phi_1(z,w)=|zw|^{\frac{p}{2}}\cos\frac{p(\pi-|t+u|)}{2}$ and 
$ z=|z|e^{\imath t}\in \mathbb{C},\, w=|w|e^{\imath u}\in \mathbb{C}.$
\end{lemma}

\begin{lemma}For $s\geq \frac{p}{p-1}$ and $p\geq 4,$ there holds the following estimate:
\begin{equation}
\label{lemma32} 
-\left(\frac{|z|^s+|w|^s}{2}\right)^{\frac{p}{s}}+\frac{|z+\overline {w}|^p }{2^p \cos^p\frac{\pi}{2p}} +\tan\frac{\pi}{2p}\Phi_2(z,w)\leq 0,
\end{equation}
where $\Phi_2(z,w)= |zw|^{\frac{p}{2}}v_p (t+u),$ $t= \arg z,$ $u= \arg w$ and
\begin{equation*} v_p (t)  =
\begin{cases}
-\cos\frac{pt}{2}, & \mbox{if}\quad  0\leq |t|\leq\frac{2\pi}{p}; \\
-\cos\frac{p(2\pi-t)}{2}, & \mbox{if}\quad 2\pi-\frac{2\pi}{p}\leq |t|\leq 2\pi;\\
\max\{|\cos\frac{pt}{2}|, |\cos\frac{p(2\pi-t)}{2}|\}, & \mbox{if}\quad  \frac{2\pi}{p} \leq |t|\leq 2\pi-\frac{2\pi}{p}.
\end{cases}
\end{equation*}
\end{lemma}

\begin{proof}[Proof of the First Part of Theorem 1.1]
Applying Lemma 3.1. for the $z=P_+f(\zeta)$ and $w=\overline{P_- f(\zeta)}$, we obtain
\begin{equation*}\begin{split}
&-\left\{\frac{|P_+f(\zeta)|^s+|P_-f(\zeta)|^s}{2}\right\}^{\frac ps}+\frac{|P_+f(\zeta)+\overline{P_-f(\zeta)}|^p}
{2^p\sin^p\frac {\pi}{2p}}+\cot\frac{\pi}{2p} \Phi_1(P_+f(\zeta)P_-f(\zeta))^{ \frac{p}{2}} \leq 0,
\end{split}\end{equation*}
i.e.,
\begin{equation*}\begin{split}
\left(|P_+f(\zeta)|^s+|P_-f(\zeta)|^s\right)^{\frac ps}&\geq \frac{2^{\frac {p}{s}}}{2^p\sin^p\frac {\pi}{2p}}|f(\zeta)|^p
+2^{\frac{p}{s}}\cot\frac{\pi}{2p} \Phi_1(P_+f(\zeta)P_-f(\zeta))^{\frac p2}\leq 0.
\end{split}\end{equation*}
If we integrate the above inequality and use
\begin{equation*}
\int_{\mathbb{T}}  \Phi_1(P_+f(\zeta)P_-f(\zeta))^{\frac{p}{2}}|d\zeta| \geq 0,
\end{equation*}
(because of plurisubharmonicity of $\Phi_1$ proved in \cite{KALAJ.TAMS}), we obtain
\begin{equation*}\begin{split}
\int_{\mathbb{T}} (|P_+ f(\zeta)|^s+|P_-f(\zeta)|^s)^{\frac ps}|d\zeta| \geq \frac{ 2^{\frac ps} }{ 2^p\sin^p\frac {\pi}{2p} }
\int_\mathbb{T} |f(\zeta)|^p|d\zeta|,
\end{split}\end{equation*}
which is exactly what was to be proved.
\end{proof}

\begin{proof}[Proof of the Second Part of Theorem 1.1]
From the paper \cite{KALAJ.TAMS} we know that $\Phi_2(z,w)$ is plurisubharmonic on
$\mathbb{C}\times\mathbb{C}$. Hence, integrating the inequality from the Lemma 3.2 for $z=P_+f(\zeta)$ and $w=\overline{P_-f(\zeta)},$ we have
\begin{equation*}\begin{split}
\int_{\mathbb{T}}\left(|P_+f(\zeta)|^s+|P_-f(\zeta)|^s\right)^{\frac ps}|d\zeta|&\geq \frac{2^{\frac {p}{s}}}{2^p\cos^p\frac{\pi}{2p}}\int_{\mathbb{T}}|f(\zeta)|^p|d\zeta|
\\&+2^{\frac{p}{s}}\tan\frac{\pi}{2p}\int_{\mathbb{T}}\phi_p (P_+f(\zeta)P_-f(\zeta))|d\zeta|.
\end{split}\end{equation*}
and
$$\int_{\mathbb{T}}\Phi_2(P_+f(\zeta)P_-f(\zeta))|d\zeta| \geq 0,$$
thus finally getting:
\begin{equation*}\begin{split}
\int_{\mathbb{T}}\left(|P_+f(\zeta)|^s+|P_-f(\zeta)|^s\right)^{\frac ps}|d\zeta|&\geq \frac{2^{\frac {p}{s}}}{2^p\cos^p\frac {\pi}{2p}}\int_{\mathbb{T}}|f(\zeta)|^p|d\zeta|.
\end{split}\end{equation*}
\end{proof}

\section{Analysis of stationary points and reduction to boundary points}

As we can see from the previous section, our proof of Theorem 1.1 strongly depends on certain inequalities for complex numbers which serves as `'benchguards'' for $P_+f$ and $P_-f.$ Here, we will give the main part of their proofs postponing ''only'' case where one of variables is equal to 0. In what follows, we sketch our approach (similar to \cite{PETARMARIJAN}) and describe how our problem reduces to certain inequalities which we prove in the last two sections. \\

First, both inequalities \eqref{lemma31} and \eqref{lemma32} are homogeneous, hence it is enough to prove them for $z=1$ nad $w=re^{\imath t}.$ We can also suppose that $0\leq r\leq 1,$ since by symmetry, we can take that $z=1$ is the variable of non-smaller modulus.  
Let us now concentrate on \eqref{lemma31}. By power-mean inequality, it is enough to show it for $s=p.$\\

Denote
$$\Phi(r,t)=\frac{(1+r^2+2r\cos t)^{\frac{p}{2}}}{2^p\sin^p\frac{\pi}{2p}}-\bigg(\frac{1+r^s}{2}\bigg)^{\frac{p}{s}}+r^{\frac{p}{2}}\cot\frac{\pi}{2p}\cos\frac{(\pi-t)p}{2}.$$
We intend to prove:
\begin{equation}
\label{svod1}
\Phi(r,t)\leq 0,
\end{equation}
from which, by the above comments, the inequality \eqref{lemma31} follows. 
Finding the derivatives of $\Phi,$ we easily infer that $\frac{\partial \Phi}{\partial r}=\frac{\partial \Phi}{\partial t}=0$ is equivalent with the system:
$$ \frac{(1+r^2+2r\cos t)^{\frac{p}{2}-1}}{2^p\sin^p\frac{\pi}{2p}}=\frac{r^{\frac{p}{2}-1}\cot\frac{\pi}{2p}\sin\frac{(\pi-t)p}{2}}{2\sin t},$$
$$ \frac{(1+r^2+2r\cos t)^{\frac{p}{2}-1}(r+\cos t)}{2^p\sin^p\frac{\pi}{2p}}+\frac{r^{\frac{p}{2}-1}}{2}\cot\frac{\pi}{2p}\cos\frac{(\pi-t)p}{2}=2^{-\frac{p}{s}}r^{s-1}(1+r^s)^{\frac{p}{s}-1}.$$
We can estimate the derivative in $t$ in the following way:
\begin{align*}
\frac{2}{p}\frac{\partial \Phi}{\partial t}&=r^{\frac{p}{2}}\bigg[\frac{-2\sin t}{2^p\sin^p\frac{\pi}{2p}}\bigg(\frac{1+r^2}{r}+2\cos t\bigg)^{\frac{p}{2}-1}+\cot\frac{\pi}{2p}\sin\frac{(\pi-t)p}{2}\bigg]\\
&\geq r^{\frac{p}{2}}\bigg[\frac{-2\sin t}{2^p\sin^p\frac{\pi}{2p}}\big(2+2\cos t\big)^{\frac{p}{2}-1}+\cot\frac{\pi}{2p}\sin\frac{(\pi-t)p}{2}\bigg]\geq 0,
\end{align*}
for $t \in [0,\pi-\frac{\pi}{p}),$ thus making a reduction to the range $[\pi-\frac{\pi}{p},\pi].$
This inequality follows after substituting $\frac{t}{2}$ with $\frac{\pi-t}{2}$ and proving
$$f(t)=\frac{\sin pt}{\cos t\sin^{p-1}t}\geq   \frac{1}{\sin^p\frac{\pi}{2p}}\tan\frac{\pi}{2p},$$
for $[\frac{\pi}{2p},\frac{\pi}{2}].$\\
Since $f'(t)=\frac{g(t)}{\cos^2 t\sin^p t},$ with $g(t)= \sin pt-p\cos t\sin(p-1)t,$ for which we have $g'(t)=p(2-p)\cos t\cos(p-1)t\geq 0$ and $g(t)\geq g(\frac{\pi}{2p})=1-p\cos^2\frac{\pi}{2p}\geq 0$ (this follows from convexity of function $h(x)=\cos^2\frac{\pi x}{2}-x$ for $x\in[\frac{1}{2},1]$ and $h(\frac{1}{2})=0>h(1)=-1$). Therefore, $f$ is increasing and the estimate for $f$ follows.

From the second equation, we can express
$$(1+r^s)^{\frac{p}{s}-1}=\frac{2^{\frac{p}{s}}r^{\frac{p}{2}-s}\cot\frac{\pi}{2p}}{2\sin t}\bigg[(r+\cos t)\sin\frac{(\pi-t)p}{2}+\sin t\cos\frac{(\pi-t)p}{2}\bigg].$$
Using this with the first equation for the stationary point, we find that this is equivalent to:
\begin{align*}
\Phi(r,t)&=(1+r^2+2r\cos t)\frac{r^{\frac{p}{2}-1}\cot\frac{\pi}{2p}\sin\frac{(\pi-t)p}{2}}{2\sin t}\\
&-\frac{r^{\frac{p}{2}-s}(1+r^s)\cot\frac{\pi}{2p}}{2\sin t}\bigg[(r+\cos t)\sin\frac{(\pi-t)p}{2}+\sin t\cos\frac{(\pi-t)p}{2}\bigg]\\
&+r^{\frac{p}{2}}\cot\frac{\pi}{2p}\cos\frac{(\pi-t)p}{2}\leq 0.
\end{align*}
After cancellation of positive terms, we get:
\begin{align*}
\bigg(r+\frac{1}{r}+2\cos t\bigg)\sin\frac{(\pi-t)p}{2}&-(1+r^{-s})\bigg[(r+\cos t)\sin\frac{(\pi-t)p}{2}+\sin t\cos\frac{(\pi-t)p}{2}\bigg]\\
&+2\sin t \cos\frac{(\pi-t)p}{2}\leq 0.
\end{align*}
This is easily seen to be equivalent with
$$(1-r^{2-s})\sin\frac{(\pi-t)p}{2}+(r-r^{1-s})\sin\bigg(t+\frac{(\pi-t)p}{2}\bigg)\leq 0,$$
or
$$\frac{1-r^{2-s}}{r-r^{1-s}}\sin\frac{(\pi-t)p}{2}+\sin\bigg(t+\frac{(\pi-t)p}{2}\bigg)\geq 0.$$
Since $\frac{1-r^{2-s}}{r-r^{1-s}}$ decreases in $r$ for $s\leq 2,$
we have $\frac{1-r^{2-s}}{r-r^{1-s}}\geq 1-\frac{2}{s}$ and it is enough to prove
$$ \bigg(1-\frac{2}{s}\bigg)\sin\frac{(\pi-t)p}{2}+\sin\bigg(t+\frac{(\pi-t)p}{2}\bigg)\geq 0$$
for $\pi-\frac{\pi}{p}\leq t\leq \pi.$
Using concavity of sine function on the appropriate range, we have:
\begin{align*}
&\quad\bigg(1-\frac{2}{s}\bigg)\sin\frac{(\pi-t)p}{2}+\sin\bigg(t+\frac{(\pi-t)p}{2}\bigg)\\
&\geq \bigg(1-\frac{2}{s}\bigg)\sin\frac{(\pi-t)p}{2}+\bigg(\frac{2}{p}-1\bigg)\sin\frac{(\pi-t)p}{2}=2\bigg(\frac{1}{p}-\frac{1}{s}\bigg)\sin\frac{(\pi-t)p}{2}\geq 0,
\end{align*}
which is true for $s\geq p. $

Now, problem reduces to proving inequality on the boundary of rectangle $\Pi:=\{0\leq r\leq 1,\pi-\frac{\pi}{p} \leq t\leq \pi\}.$ It trivially holds by the construction of minorant for $t=\pi-\frac{\pi}{p}.$ 
For $r=1$ our inequality becomes (for $0\leq t\leq \pi$)
$$\cos^p\frac{\pi-t}{2}\leq \sin^p\frac{\pi}{2p}-\sin^{p-1}\frac{\pi}{2p}\cos\frac{\pi}{2p}\cos\frac{tp}{2}.$$
After substitution, we have
$$\varphi(t)=\frac{1-\cot\frac{\pi}{2p}\cos pt}{\sin^p t}\geq \frac{1}{\sin^p\frac{\pi}{2p}}, \quad \text{for} \quad 0\leq t\leq \frac{\pi}{2}.$$
From 
$$\frac{\varphi'(t)}{\varphi(t)}=\frac{ph(t)}{\sin t(1-\cot\frac{\pi}{2p}\cos pt)},$$
where $h(t)=\cot\frac{\pi}{2p}\cos((p-1)t)-\cos t.$
Since $h'(t)=\sin t-(p-1)\cot\frac{\pi}{2p}\sin((p-1)t)\geq (1-(p-1)\cot\frac{\pi}{2p})\sin((p-1)t)\geq 0,$ we infer that $h$ increases and has its only zero in $(,\frac{\pi}{2})$ in $t=\frac{\pi}{2p}.$ Therefore, $\varphi$ has its minimum in $\frac{\pi}{2p}$ and the estimate follows.

The case $r=0$ gives $2^p\sin^p\frac{\pi}{2p}\geq 2.$ This is equivalent with $\sin\frac{\pi x}{2}\geq 2^{x-1}$ with $x=\frac{1}{p}\in [\frac{1}{2},1].$ But, $\sin\frac{\pi x}{2}-2^{x-1}$ is easily seen to be concave function with values in $\frac{1}{2}$ and $1$ equal to $0$ and the estimate follows. The case $t=0$ will be discussed in the next section. \\
 
In case $p\geq 2,$ our main objective is proving (with $s=\frac{p}{p-1}$)
\begin{equation}
\label{svod2}
\Phi(r,t):=\frac{(1+r^2+2r\cos t)^{\frac{p}{2}}}{2^p\cos^p\frac{\pi}{2p}}-\bigg(\frac{1+r^s}{2}\bigg)^{\frac{p}{s}}+r^{\frac{p}{2}}\tan\frac{\pi}{2p}v_p(t)\leq 0.
\end{equation}

Since $v_p(t)\leq 1,$ for $\frac{2\pi}{p}\leq t\leq 2\pi-\frac{2\pi}{p},$ we easily get:
\begin{align*}
&\quad\frac{(1+r^2+2r\cos t)^{\frac{p}{2}}}{2^p\cos^p\frac{\pi}{2}}-\bigg(\frac{1+r^s}{2}\bigg)^{\frac{p}{s}}+r^{\frac{p}{2}}\tan\frac{\pi}{2p}v_p(t)\\
&\leq\frac{(1+r^2+2r\cos \frac{2\pi}{p})^{\frac{p}{2}}}{2^p\cos^p\frac{\pi}{2p}}-\bigg(\frac{1+r^s}{2}\bigg)^{\frac{p}{s}}+r^{\frac{p}{2}}\tan\frac{\pi}{2p}=\Phi\bigg(r,\frac{2\pi}{p}\bigg),
\end{align*}
hence it is enough to prove the inequality for $0\leq t\leq \frac{2\pi}{p}.$ The same holds for $2\pi-\frac{2\pi}{p}\leq t\leq 2\pi,$ since $\Phi(r,t)=\Phi(r,2\pi-t).$
Further, we can see that
\begin{align*}
\frac{2}{p}\frac{\partial \Phi}{\partial t}&=r^{\frac{p}{2}}\bigg[\frac{-2r\sin t}{2^p\cos^p\frac{\pi}{2p}}\bigg(\frac{1+r^2}{r}+2\cos t\bigg)^{\frac{p}{2}-1}+\tan\frac{\pi}{2p}\sin\frac{tp}{2}\bigg]\\
&\leq r^{\frac{p}{2}}\bigg[\frac{-2r\sin t}{2^p\cos^p\frac{\pi}{2p}}\big(2+2\cos t\big)^{\frac{p}{2}-1}+\tan\frac{\pi}{2p}\sin\frac{tp}{2}\bigg]=:r^{\frac{p}{2}}g(t)\\
&\leq r^{\frac{p}{2}}g\bigg(\frac{\pi}{p}\bigg)=0,
\end{align*}
for $\frac{\pi}{p}\leq t\leq \frac{2\pi}{p}.$
Last inequality follows from the observation that 
$$f(t)=\frac{\sin\frac{tp}{2}}{\sin t(1+\cos t)^{\frac{p}{2}-1}}=2^{-\frac{p}{2}}\frac{\sin\frac{tp}{2}}{\sin\frac{t}{2}\cos^{p-1}\frac{t}{2}}=2^{-\frac{p}{2}}F\bigg(\frac{t}{2}\bigg)$$
for $F(t)=\frac{\sin p t}{\sin t\cos^{p-1}t}$
is decreasing on $[\frac{\pi}{2p},\frac{\pi}{p}].$ Indeed, since
$$F'(t)=\frac{1}{\sin^2t\cos^{2p-2}t}\bigg(p\sin t\cos(p-1)t-\sin pt\bigg)=:\frac{h(t)}{\sin^2t\cos^{2p-2}t},$$
$$h'(t)=-p(p-2)\sin t\sin(p-1)t\leq 0,$$
we have $h(t)\leq h(\frac{2\pi}{p})=p\sin^2\frac{\pi}{2p}-1\leq 0$ (an easy exercise) and finally $F'(t)\leq 0.$
Now, the inequality reduces to $\Phi(r,t)\leq 0,$ for $0\leq r\leq 1$ and $0\leq t\leq \frac{\pi}{p}.$
First, we analyse potential stationary points as in the case $1<p<2.$ We have:
$$\frac{\partial \Phi}{\partial t}=0 \iff \frac{(1+r^2+2r\cos t)^{\frac{p}{2}-1}}{2^p\cos^p\frac{\pi}{2p}}=\frac{\tan\frac{\pi}{2p}\sin\frac{tp}{2}r^{\frac{p}{2}-1}}{2\sin t}.$$
and 
$$\frac{2}{p}\frac{\partial \Phi}{\partial r}=0 \iff (1+r^s)^{\frac{p}{s}-1}=\frac{2^{\frac{p}{s}-1}r^{\frac{p}{2}-s}\tan\frac{\pi}{2p}}{\sin t}\bigg(-\sin t\cos\frac{tp}{2}+(r+\cos t)\sin\frac{tp}{2}\bigg).$$
In such a point, the function $\Phi(r,t)$ is equal to 
\begin{align*}
\Phi(r,t)&=\frac{r^{\frac{p}{2}-1}(1+r^2+2r\cos t)\tan\frac{\pi}{2p}\sin\frac{tp}{2}}{2\sin t}\\
&-\frac{(1+r^s)r^{\frac{p}{2}-s}\tan\frac{\pi}{2p}}{\sin t}\bigg(-\sin t\cos\frac{tp}{2}+(r+\cos t)\sin\frac{tp}{2}\bigg)-r^{\frac{p}{2}}\tan\frac{\pi}{2p}\cos\frac{tp}{2}.
\end{align*}
Non-positivity of this expression for $0<t<\frac{\pi}{p}$ is equivalent with
$$\frac{1-r^{2-s}}{r-r^{1-s}}\sin\frac{tp}{2}+\sin\bigg(\frac{p}{2}-1\bigg)t\geq 0.$$
We know that $\frac{1-r^{2-s}}{r-r^{1-s}}\geq 1-\frac{2}{s}$, hence:
\begin{align*}
&\quad\frac{1-r^{2-s}}{r-r^{1-s}}\sin\frac{tp}{2}+\sin\bigg(\frac{p}{2}-1\bigg)t\\
&\geq \bigg(1-\frac{2}{s}\bigg)\sin\frac{tp}{2}+\sin\bigg(\frac{p}{2}-1\bigg)t\\
&\geq \bigg(1-\frac{2}{s}\bigg)\sin\frac{tp}{2}+\bigg(1-\frac{2}{p}\bigg)\sin\frac{tp}{2}=2\bigg(1-\frac{1}{p}-\frac{1}{s}\bigg)\sin\frac{tp}{2}\geq 0
\end{align*}
for $s\geq\frac{p}{p-1}$ and $0<t<\frac{\pi}{p}.$

Inequality \eqref{lemma32} for $t=\frac{\pi}{p}$ follows from the discussion given in the second section. Case $t=0$ will be proved in the sixth section. If $r=0,$ inequality follows from $2\cos^p\frac{\pi}{2p}\geq 2\cos^2\frac{\pi}{2p}=1+\cos\frac{\pi}{p}\geq 1.$ Case $r=1$ reduces to the inequality:
$$\frac{\cos^p t}{\cos^p\frac{\pi}{2p}}-1-\tan\frac{\pi}{2p}\cos pt\leq 0$$
for $0\leq t\leq \frac{\pi}{p}.$
We will omit the proof of this inequality since it is very similar to the appropriate estimate in case $1<p<2.$\\

In next two sections we will complete our proofs of \eqref{lemma31} and \eqref{lemma32} by proving the remaining cases of \eqref{svod1} and \eqref{svod2},i.e. those when $t=0.$

\section{A proof of inequality $(3)$ for $t=\pi$}

In this section, our aim is to prove 
$$\frac{(1-r)^p}{2^p\sin^p\frac{\pi}{2p}}-\frac{1+r^p}{2}+r^{\frac{p}{2}}\cot\frac{\pi}{2p}\leq 0,$$
for $1<p<2.$

Using Jensen's inequality 
$$(c+1)\bigg[\frac{a^{\frac{p}{2}}}{c+1}+\frac{c}{c+1}b^{\frac{p}{2}}\bigg]\leq (c+1)\bigg(\frac{a}{c+1}+\frac{bc}{c+1}\bigg)^{\frac{p}{2}}, \quad \frac{p}{2}\leq 1,$$
with $a=\frac{(1-r)^2}{4\sin^2\frac{\pi}{2p}},$ $b=r$ and $c=\cot\frac{\pi}{2p},$
we get:
$$\frac{(1-r)^p}{2^p\sin^p\frac{\pi}{2p}}+r^{\frac{p}{2}}\cot\frac{\pi}{2p}\leq \big(1+\cot\frac{\pi}{2p}\big)^{1-\frac{p}{2}}\bigg(\frac{(1-r)^2}{4\sin^2\frac{\pi}{2p}}+r\cot\frac{\pi}{2p}\bigg)^{\frac{p}{2}}.$$
We have to prove 
$$L(r)=\big(1+\cot\frac{\pi}{2p}\big)^{1-\frac{p}{2}}\frac{\big((1-r)^2+2r\sin\frac{\pi}{p}\big)^{\frac{p}{2}}}{2^p\sin^p\frac{\pi}{2p}}\leq \frac{1+r^p}{2},$$
or
$$F(r)=\frac{\big((1-r)^2+2r\sin\frac{\pi}{p}\big)^{\frac{p}{2}}}{1+r^p}\leq 2^{p-1}\sin^p\frac{\pi}{2p}\bigg(1+\cot\frac{\pi}{2p}\bigg)^{-1+\frac{p}{2}}.$$
From
$$F'(r)=\frac{p\big((1-r)^2+2r\sin\frac{\pi}{p}\big)^{\frac{p}{2}-1}}{(1+r^p)^2}\bigg[r-r^{p-1}+(r^p-1)(1-\sin\frac{\pi}{p})\bigg],$$
we see that $F$ decreases and $F(r)\leq F(0)=1,$
therefore the problem reduces to proving 
\begin{equation}
\label{prvisl}
\big(1+\cot\frac{\pi}{2p}\big)^{1-\frac{p}{2}}\leq 2^{p-1}\sin^p\frac{\pi}{2p}.
\end{equation}
We will prove this inequality later, for $\frac{4}{3}\leq p\leq 2.$\\

For the rest of the interval, we use convexity of the power function. Namely, from $(1-r)^p\leq 1-r^p,$ we have:
\begin{align*}
&\quad\frac{(1-r)^p}{2^p\sin^p\frac{\pi}{2p}}-\frac{1+r^p}{2}+r^{\frac{p}{2}}\cot\frac{\pi}{2p}\\
&\leq \frac{1-r^p}{2^p\sin^p\frac{\pi}{2p}}-\frac{1+r^p}{2}+r^{\frac{p}{2}}\cot\frac{\pi}{2p}\\
&=r^p\bigg(-\frac{1}{2}-\frac{1}{2^p\sin^p\frac{\pi}{2p}}\bigg)+r^{\frac{p}{2}}\cot\frac{\pi}{2p}+\frac{1}{2^p\sin^p\frac{\pi}{2p}}-\frac{1}{2}.
\end{align*}
The last expression is a quadratic function on $r^{\frac{p}{2}}$ and it attains the maximum in the point $0\leq\frac{\cot\frac{\pi}{2p}}{1+\frac{2}{2^p\sin^p\frac{\pi}{2p}}}<1.$
Therefore, we need to prove:
\begin{equation}
\label{drugisl}
1-\frac{4}{4^p\sin^{2p}\frac{\pi}{2p}}\geq \cot^2\frac{\pi}{2p}.
\end{equation}
This inequality will be proved for $1<p\leq\frac{4}{3}.$\\

Let us prove that $\big(1+\cot\frac{\pi}{2p}\big)^{1-\frac{p}{2}}\bigg(\frac{1}{\sin^2\frac{\pi}{2p}}\bigg)^{\frac{p}{2}}\leq 2^{p-1},$ for $\frac{4}{3}\leq p\leq 2$ (this is rephrased \eqref{prvisl}). From weighted AM-GM inequality, we have: 
\begin{align*}
\big(1+\cot\frac{\pi}{2p}\big)^{1-\frac{p}{2}}\bigg(\frac{1}{\sin^2\frac{\pi}{2p}}\bigg)^{\frac{p}{2}}&\leq \big(1-\frac{p}{2}\big)\big(1+\cot\frac{\pi}{2p}\big)+\frac{p}{2}\frac{1}{\sin^2\frac{\pi}{2p}}\\
&=\frac{p}{2}\cot^2\frac{\pi}{2p}+\big(1-\frac{p}{2}\big)\cot\frac{\pi}{2p}+1.
\end{align*}
Let us prove that $F(p)=\frac{p}{2}\cot^2\frac{\pi}{2p}+(1-\frac{p}{2})\cot\frac{\pi}{2p}\leq 2^{p-1}-1.$
We easily find that 
$$F''(p)=\frac{\pi p}{8p^4\sin^4\frac{\pi}{2p}}\bigg[4(\pi-1)+\pi\bigg(\frac{2}{p}-1\bigg)\sin\frac{\pi}{p}+(2\pi+4)\cos\frac{\pi}{p}\bigg].$$
Denote $g(x)=4\pi-4+\pi(2x-1)\sin\pi x+(2\pi+4)\cos\pi x$ with $x=\frac{1}{p}\in [\frac{1}{2},\frac{3}{4}].$ Since $g'(x)=\pi\sin\pi x\big(-2(1+\pi)+\pi(2x-1)\cot\pi x\big)<0,$  $g$ is decreasing and, consequently, $g(\frac{1}{p})$ increasing in $p.$ Therefore, $g(x)\geq g(\frac{3}{4})=(4-\frac{3}{2\sqrt{2}})\pi-4-2\sqrt{2}>0$ and $F$ is convex. For $\frac{4}{3}\leq p\leq \frac{3}{2},$ this implies:
$$F(p)=F\bigg((1-\alpha)\frac{4}{3}+\alpha\frac{3}{2}\bigg)\leq (1-\alpha)F\bigg(\frac{4}{3}\bigg)+\alpha F\bigg(\frac{3}{2}\bigg)=(1-\alpha)\big(\frac{5}{3}-\sqrt{2}\big)+\frac{\alpha}{4}\big(1+\frac{1}{\sqrt{3}}\big),$$
and, since $\alpha=6p-8,$ we get:
$$F(p)\leq (9-6p)\bigg(\frac{5}{3}-\sqrt{2}\bigg)+\frac{3p-4}{2}\bigg(1+\frac{1}{\sqrt{3}}\bigg).$$
Note that, for 
$$G_1(p)=2^{p-1}-1-(9-6p)\bigg(\frac{5}{3}-\sqrt{2}\bigg)-\frac{3p-4}{2}\bigg(1+\frac{1}{\sqrt{3}}\bigg),$$
we have
\begin{align*}
G_1'(p)&=\frac{1}{2}\big(2^p\log 2+17-12\sqrt{2}-\sqrt{3}\big)\\
&\geq  \frac{1}{2}\big(2^{\frac{4}{3}}\log 2+17-12\sqrt{2}-\sqrt{3}\big)> 0,
\end{align*}
hence $G_1$ increases and $G_1(p)\geq G_1(\frac{4}{3})=\sqrt[3]{2}+\sqrt{2}-\frac{8}{3}>0.$
This concludes the proof of $F(p)\leq 2^{p-1}-1$ in this case. \\
We work similarly for $\frac{3}{2}\leq p\leq \frac{8}{5},$ but we represent $p$ as a convex combination of $\frac{3}{2}$ and $2:$
$$F(p)=F\bigg((1-\alpha)\frac{3}{2}+2\alpha\bigg)\leq (1-\alpha)F\bigg(\frac{3}{2}\bigg)+\alpha F(2)=\frac{1-\alpha}{4}\bigg(1+\frac{1}{\sqrt{3}}\bigg)+\alpha,$$
thus getting:
$$F(p)\leq (4-2p)\bigg(\frac{1}{4}+\frac{1}{4\sqrt{3}}\bigg) +2p-3.$$
We easily see that the function $G_2(p)=2^{p-1}-1-(4-2p)\big(\frac{1}{4}+\frac{1}{4\sqrt{3}}\big) -2p+3$ is monotone decreasing, since
$$G_2'(p)=2^{p-1}\log 2-\frac{3}{2}+\frac{1}{2\sqrt{3}}< 2^{\frac{3}{5}}\log 2-\frac{3}{2}+\frac{1}{2\sqrt{3}}<0,$$
which implies $G_2(p)\geq G_2(\frac{8}{5})=2^{\frac{3}{5}}\log 2-\frac{1}{5}\big(7+\frac{1}{\sqrt{3}}\big)>0,$
again, getting $F(p)\leq 2^{p-1}-1.$\\
Since $(F(p)-2^{p-1}+1)''=-2^{p-1}\log^2 2+\frac{\pi p}{8p^4\sin^4\frac{\pi}{2p}}\big[4(\pi-1)+\pi(\frac{2}{p}-1)\sin\frac{\pi}{p}+(2\pi+4)\cos\frac{\pi}{p}\big]$
and $p\sin\frac{\pi}{2p}$ increases in $p,$
we have
\begin{align*}
\big(F(p)-2^{p-1}+\big)''&\geq -2^{p-1}\log^2 2+\frac{\pi p}{32}\bigg[4(\pi-1)+\pi\bigg(\frac{2}{p}-1\bigg)\sin\frac{\pi}{p}+(2\pi+4)\cos\frac{\pi}{p}\bigg]\\
&=\frac{p\pi(\pi-1)}{8}-2^{p-1}\log^2 2+\frac{p\pi}{32}\bigg[\pi\bigg(\frac{2}{p}-1\bigg)\sin\frac{\pi}{p}+(2\pi+4)\cos\frac{\pi}{p}\bigg].
\end{align*}
The last summand is increasing in $p,$ by the observation from the first part of the proof and is not smaller than its value in $\frac{8}{5},$ while 
$$\bigg(\frac{p\pi(\pi-1)}{8}-2^{p-1}\log^2 2\bigg)'=\frac{\pi(\pi-1)}{8}-2^{p-1}\log^3 2\geq \frac{\pi(\pi-1)}{8}-2\log^3 2 >0,$$
hence 
$$\big(F(p)-2^{p-1}+1\big)''\geq \frac{\pi(\pi-1)}{5}-2^{\frac{3}{5}}\log^2 2+\frac{\pi}{20}\bigg(\frac{\pi}{4}\sin\frac{5\pi}{8}+(2\pi+4)\cos\frac{5\pi}{8}\bigg)>0.$$
Therefore $F(p)-2^{p-1}+1$ is concave and attains its maximum on $[\frac{8}{5},2]$ in some of the endpoints. We have already proved that $F(\frac{8}{5})-1+2^{\frac{3}{5}}<0=F(2)-1,$ thus, $F(p)\leq 2^{p-1}-1$ on the whole interval $[\frac{4}{3},2].$\\

By using the change of variable $x=\cot\frac{\pi}{2p}$ the inequality \eqref{drugisl}
transforms into
$$\frac{1-x^2}{4}\geq \bigg(\frac{1+x^2}{4}\bigg)^{\frac{\pi}{2\arccot x}}$$
or
$$\log\frac{1-x^2}{4}\geq \frac{\pi}{2\arccot x}\log\frac{1+x^2}{4},$$
for $0\leq x\leq \cot\frac{3\pi}{8}.$
Since
$$\frac{d^2}{d x^2}\bigg(\frac{1}{\arccot x}\bigg)=-2\frac{x(\arctan{\frac{1}{x}}-\frac{1}{x})}{(x^2+1)^2 \arctan^3\frac{1}{x}},$$
we infer that $\frac{\pi}{2\arccot x}$ is convex and:
\begin{align*}
f(x)&=\log\frac{1-x^2}{4}-\frac{\pi}{2\arccot x}\log\frac{1+x^2}{4}\\
&\geq \log\frac{1-x^2}{4}-\bigg(1+\frac{2}{\pi}x\bigg)\log\frac{1+x^2}{4}=:g_1(x),
\end{align*}
where $1+\frac{2}{\pi}x$ is the tangent ofthe graph of $\frac{\pi}{2\arccot x}$ at the point $(0,1).$ 
From
$$g_1''(x)=-\frac{4(3\pi x^4+(x^2-1)^2(x^2+3)x+\pi)}{\pi(1-x^4)^2}<0,$$
we see that $g_1$ is concave and:
$$f(x)\geq g_1(x)\geq \min\bigg\{g_1(0),g_1\bigg(\frac{1}{3}\bigg)\bigg\}=\min\bigg\{0,\bigg(1+\frac{2}{3\pi}\bigg)\log\frac{18}{5}-\log\frac{9}{2}\bigg\}=0,$$
since $(1+\frac{2}{3\pi})\log\frac{18}{5}-\log\frac{9}{2}>\frac{6}{5}\log\frac{18}{5}-\log\frac{9}{2}=\frac{1}{5}\log\frac{18432}{15625}>0.$ Therefore, $f(x)\geq 0,$ for $x \in [0,\frac{1}{3}].$\\
For $[\frac{1}{3},\cot\frac{3\pi}{8}]$ we use
\begin{align*}
\frac{\pi}{2\arccot x}&\geq \frac{\pi}{2\arccot \frac{1}{3}}+\frac{\pi\big(x-\frac{1}{3}\big)}{2(1+\frac{1}{9})\arccot \frac{1}{3}}\\
&\geq \frac{5}{4}+\frac{45}{56}\bigg(x-\frac{1}{3}\bigg).
\end{align*}
We used here that $\frac{\pi}{2\arccot \frac{1}{3}}\geq \frac{5}{4},$ which is equivalent to $\frac{1}{3}\geq \cot\frac{2\pi}{5}=\sqrt{1-\frac{2}{\sqrt{5}}}.$
From this, we get:
\begin{align*}
f(x)&\geq g_2(x)=\log\frac{1-x^2}{4}-\frac{45x+55}{56}\log\frac{1+x^2}{4}\\
&\geq\min\bigg\{g_2\bigg(\frac{1}{3}\bigg),g_2\bigg(\cot\frac{3\pi}{8}\bigg)\bigg\}\\
&=\min\bigg\{\frac{5}{4}\log\frac{18}{5}-\log\frac{9}{2},-\frac{10+45\sqrt{2}}{56}\log\frac{2-\sqrt{2}}{2}+\log\frac{\sqrt{2}-1}{2}\bigg\}>0,
\end{align*}
since
$$g_2''(x)=-\frac{1}{28(1-x^4)^2}\bigg[45x^7+x^6+45x^5+333x^4-225x^3+3x^2+135x+111\bigg]<0$$
for $0\leq x\leq 1.$ Finally, $f(x)\geq 0$ for $x \in [\frac{1}{3},\cot\frac{3\pi}{8}].$

\section{A proof of inequality $(4)$ for $t=0$}

In this section, we will prove 
$$\Phi(r)=\frac{(1+r)^p}{2^p\cos^p\frac{\pi}{2p}}-\bigg(\frac{1+r^{\frac{p}{p-1}}}{2}\bigg)^{p-1}-r^{\frac{p}{2}}\tan\frac{\pi}{2p}\leq 0$$
for $p\geq 4.$

Changing the variable as $r=e^{-2y}$ and multiplying both sides by $e^{py},$ we get the following equivalent form:
$$F(y)=\cosh^{p-1}\frac{py}{p-1}+\tan\frac{\pi}{2p}-\frac{\cosh^py}{\cos^p\frac{\pi}{2p}}\geq 0.$$
It can be easily calculated that $F'(y)=p G(y)\cosh^{p-1}y\sinh y,$ where
$$G(y)=\frac{\cosh^{p-2}\frac{py}{p-1}\sinh\frac{py}{p-1}}{\cosh^{p-1}y\sinh y}-\frac{1}{\cos^p\frac{\pi}{2p}}.$$
This function increases in $y,$ since
$$\frac{4(p-1)\cosh^py\sinh^2y}{\cosh^{p-3}\frac{py}{p-1}}G'(y)=(p-1)(p-2)\sinh\frac{2py}{p-1}-p(p-3)\sinh(2y)-p(p-1)\sinh\frac{2y}{p-1}\geq 0.$$
The last expression can be represented via infinite sum as
$$\sum_{k=1}^{+\infty}\frac{(2y)^{2k-1}}{(2k-1)!(p-1)^{2k-1}}\bigg((p-1)(p-2)p^{2k-1}-p(p-3)(p-1)^{2k-1}-p(p-1)\bigg),$$
where every coefficient is positive, since 
\begin{align*}
&\quad(p-1)(p-2)p^{2k-1}-p(p-3)(p-1)^{2k-1}-p(p-1)\\
&=(p^2-3p+2)\big(p^{2k-1}-(p-1)^{2k-1}\big)+2(p-1)^{2k-1}-p(p-1)\\
&\geq p^2-3p+2+2(p-1)-p(p-1)=0.
\end{align*}

Now, we prove that $G\big(\frac{\pi}{4\sqrt{p}}\big)>0.$ Namely, there holds the following: 
\begin{lemma}
For $p\geq 4$ there holds the following estimate
$$\frac{\cosh^{p-2}\frac{\pi\sqrt{p}}{4(p-1)}\sinh\frac{\pi\sqrt{p}}{4(p-1)}}{\cosh^{p-1}\frac{\pi}{4\sqrt{p}}\sinh\frac{\pi}{4\sqrt{p}}}\geq \frac{1}{\cos^p\frac{\pi}{2p}}.$$
\end{lemma}
\begin{proof}
Note that $g(x)=\frac{\sinh x}{x+\frac{x^3}{6}}$ increases in $x,$
since $g'(x)=\frac{x}{(x+\frac{x^3}{6})^2}\big[(1+\frac{x^2}{6})\cosh x-\frac{\sinh x}{x}(1+\frac{x^2}{2}\big]$ is positive as it is the expression inside the angular brackets. Indeed, 
$$\bigg(1+\frac{x^2}{6}\bigg)\cosh x=1+\sum_{k=1}^{+\infty}\bigg(\frac{1}{(2k)!}+\frac{1}{6\cdot(2k-2)!}\bigg)x^{2k}$$
and $$\bigg(1+\frac{x^2}{2}\bigg)\frac{\sinh x}{x}=1+\sum_{k=1}^{+\infty}\bigg(\frac{1}{(2k+1)!}+\frac{1}{2\cdot(2k-1)!}\bigg)x^{2k}$$
and $\frac{1}{(2k)!}+\frac{1}{6\cdot(2k-2)!}\geq \frac{1}{(2k+1)!}+\frac{1}{2\cdot(2k-1)!}.$ (After multiplying by (2k+1)! this reduces to $2k^2-3k+1\geq 0.)$
Using this, we get
$$\frac{\sinh\frac{\pi\sqrt{p}}{4(p-1)}}{\sinh\frac{\pi}{4\sqrt{p}}}\geq \frac{\frac{\pi\sqrt{p}}{4(p-1)}\big(1+\frac{\pi^2 p}{96(p-1)^2}\big)}{\frac{\pi}{4\sqrt{p}}\big(1+\frac{\pi^2}{96p}\big)}=\frac{p}{p-1}\frac{1+\frac{\pi^2p}{96(p-1)^2}}{1+\frac{\pi^2}{96p}}.$$
Hence, we are reduced to prove 
$$\frac{\cosh^{p-2}\frac{\pi\sqrt{p}}{4(p-1)}}{\cosh^{p-1}\frac{\pi}{4\sqrt{p}}}\geq \frac{p-1}{p}
\frac{1+\frac{\pi^2}{96p}}{1+\frac{\pi^2p}{96(p-1)^2}}\frac{1}{\cos^p\frac{\pi}{2p}}.$$
We have $(\frac{\cosh x}{1+\frac{x^2}{2}})'=\frac{x}{\cosh^2x}\big(-\cosh x+(1+\frac{x^2}{2})\frac{\sinh x}{x}\big)\geq 0,$
since $(1+\frac{x^2}{2})\frac{\sinh x}{x}=1+\sum_{k=1}^{+\infty}\big(\frac{1}{(2k+1)!}+\frac{1}{2\cdot(2k-1)!}\big)x^{2k}\geq 1+\sum_{k=1}^{+\infty}\frac{x^{2k}}{(2k)!}=\cosh x,$
therefore:
$$\frac{\cosh^{p-2}\frac{\pi\sqrt{p}}{4(p-1)}}{\cosh^{p-1}\frac{\pi}{4\sqrt{p}}}\geq\frac{1}{\cosh\frac{\pi}{4\sqrt{p}}}\bigg(\frac{1+\frac{\pi^2p}{32(p-1)^2}}{1+\frac{\pi^2}{32p}}\bigg)^{p-2}.$$
Now, it remains to show:
$$\frac{1}{\cosh\frac{\pi}{4\sqrt{p}}}\bigg(\frac{1+\frac{\pi^2p}{32(p-1)^2}}{1+\frac{\pi^2}{32p}}\bigg)^{p-2}\cdot\frac{p}{p-1}\cdot\frac{1+\frac{\pi^2p}{96(p-1)^2}}{1+\frac{\pi^2}{96p}}\cos^p\frac{\pi}{2p}\geq 1$$
or
\begin{align*}
L(p)&:=\log\frac{p}{p-1}+p\log(\cos\frac{\pi}{2p})+(p-2)\bigg[\log\bigg(1+\frac{\pi^2p}{32(p-1)^2}\bigg)-\log\bigg(1+\frac{\pi^2}{32p}\bigg)\bigg]\\
&+\log\bigg(1+\frac{\pi^2p}{96(p-1)^2}\bigg)-\log\bigg(1+\frac{\pi^2}{96p}\bigg)-\log\cosh\frac{\pi}{4\sqrt{p}}\geq 0.
\end{align*}
We use the following estimates:
\begin{align*}&\log\frac{p}{p-1}=\log\frac{1}{1-\frac{1}{p}}\geq \frac{1}{p}+\frac{1}{2p^2}, \\
&p\log(\cos\frac{\pi}{2p})\geq p\log\bigg(1-\frac{\pi^2}{8p^2}\bigg)=-p\sum_{k=1}^{+\infty}\frac{\pi^{2k}}{kb^kp^{2k}}\\
&\geq -p\cdot\frac{\pi^2}{8p^2}-\frac{p}{2}\cdot\frac{\frac{\pi^4}{64p^4}}{1-\frac{\pi^2}{8p^2}}=-\frac{\pi^2}{8p}-\frac{\pi^4}{16p(8p^2-\pi^2)},\\
&\cosh\frac{\pi}{4\sqrt{p}}\leq e^{\frac{\pi^2}{32p}},\\
&\log\bigg(1+\frac{\pi^2p}{32(p-1)^2}\bigg)-\log\bigg(1+\frac{\pi^2}{32p}\bigg)=\bigg(\frac{\pi^2p}{32(p-1)^2}-\frac{\pi^2}{32p}\bigg)\cdot\frac{1}{1+\xi}\\
&\geq \frac{\pi^2(2p-1)}{32p(p-1)^2}\cdot\frac{1}{1+\frac{\pi^2p}{32(p-1)^2}}=\frac{\pi^2(2p-1)}{p(32(p-1)^2+\pi^2p)},\\
&\log\bigg(1+\frac{\pi^2p}{96(p-1)^2}\bigg)-\log\bigg(1+\frac{\pi^2}{96p}\bigg)\geq \frac{\pi^2(2p-1)}{p(96(p-1)^2+\pi^2p)}.
\end{align*}
Note that last two inequalities are implied by Lagrange's mean value theorem and \\
monotonicity of the derivative of the logarithm.\\
We have:
$$p L(p)\geq 1 +\frac{1}{2p}-\frac{\pi^2}{8}-\frac{\pi^4}{16(8p^2-\pi^2)}+\frac{\pi^2(p-2)(2p-1)}{32(p-1)^2+\pi^2p}+\frac{\pi^2(2p-1)}{96(p-1)^2+\pi^2p}-\frac{\pi^2}{32}.$$
Since
\begin{align*}
-&\frac{\pi^4}{16(8p^2-\pi^2)}\geq-\frac{\pi^4}{112p^2},\\
&\frac{\pi^2(2p-1)}{96(p-1)^2+\pi^2p}\geq \frac{\pi^2(2p-2)}{102(p-1)^2}=\frac{\pi^2}{51(p-1)}
\end{align*}
and
$$\frac{\pi^2(2p-1)(p-2)}{32(p-1)^2+\pi^2p}=\frac{\pi^2}{16}-\frac{\pi^2 p(1+\frac{\pi^2}{16})}{32(p-1)^2+\pi^2p}$$
the estimate $L(p)\geq 0$ follows from
$$1-\frac{3\pi^2}{32}-\frac{\pi^4}{112p^2}+\frac{1}{2p}+\frac{\pi^2}{51(p-1)}-\frac{\pi^2p\big(1+\frac{\pi^2}{16}\big)}{32(p-1)^2+\pi^2p}\geq 0.$$
Introducing $x=\frac{1}{p},$ we have
$$H(x)=1-\frac{3\pi^2}{32}-\frac{\pi^4x^2}{112}+\frac{x}{2}+\frac{\pi^2x}{51(1-x)}-\frac{\pi^2\big(1+\frac{\pi^2}{16}\big)x}{32(1-x)^2+\pi^2x}\geq 0.$$
But, this function is concave, as
$$H''(x)=-\frac{\pi^4}{56}+\frac{2\pi^2}{51(1-x)^3}-\frac{4\pi^2(16+\pi^2)(32x^3-96x-\pi^2+64)}{(32(1-x)^2+\pi^2x)^3}\leq 0.$$
(The last summand is evidently negative, while the sum of the first two has its maximum in $x=\frac{1}{4},$ which is negative.) This means that $H(x)\geq\min\{H(0),H(\frac{1}{4})\}>0.$\\

Therefore, it is enough to prove $F(y)\geq 0$ for $0\leq y\leq \frac{\pi}{4\sqrt{p}}.$ In fact, we will prove slightly stronger estimate:
$$f(y):=\bigg(1+\frac{p^2y^2}{2(p-1)^2}\bigg)^{p-1}+\tan\frac{\pi}{2p}-\frac{\cosh^p y}{\cos^p\frac{\pi}{2p}}\geq 0, \quad 0\leq y\leq \frac{\pi}{4\sqrt{p}}.$$
We use $\cosh \frac{py}{p-1}\geq 1+\frac{p^2y^2}{2(p-1)^2}.$\\
First, we prove that $f$ is decreasing. Namely, 
$$f'(y)=\frac{p^2y}{p-1}\bigg(1+\frac{p^2y^2}{2(p-1)^2}\bigg)^{p-2}-\frac{p\cosh^{p-1}y\sinh y}{\cos^p\frac{\pi}{2p}}$$
and $f'(y)\leq 0$ follows from 
$$A_p=\frac{p\cos^p\frac{\pi}{2p}}{p-1}\leq \frac{\cosh^{p-1}y\sinh y}{y\big(1+\frac{p^2y^2}{2(p-1)^2}\big)^{p-2}}.$$
We have
$$\cosh^{p-1} y\sinh y\geq \big(1+\frac{y^2}{2}\big)^{p-1}\sinh y\geq \big(1+\frac{y^2}{2}\big)^{p-1}\big(y+\frac{y^3}{6}\big),$$
which is $\geq A_py\big(1+\frac{p^2y^2}{2(p-1)^2}\big)^{p-2}$ if and only if
$$g(t):=(p-1)\log\bigg(1+\frac{t}{2}\bigg)+\log\bigg(1+\frac{t}{6}\bigg)-(p-2)\log\bigg(1+\frac{p^2t}{2(p-1)^2}\bigg)\geq \log A_p, \quad 0\leq t\leq \frac{\pi^2}{16p}.$$
But,
$$g'(t)=\frac{2(p^2t^2+(4p^2+p)t-4p^2+14p-4)}{(t+2)(t+6)(p^2t+2(p-1)^2)}$$
and from 
\begin{align*}&2\big(p^2t^2+(4p^2+p)t-4p^2+14p-4\big)\leq \frac{\pi^4}{256}+\frac{\pi^2}{16p}\big(4p^2+p\big)-4p^2+14p-4\\
&=-4p^2+p\bigg(14+\frac{\pi^2}{4}\bigg)+\frac{\pi^4}{256}+\frac{\pi^2}{16}-4=:l(p),
\end{align*}
we see that $l(p),$ as a quadratic polynomial, takes its maximal value on $[4,+\infty)$ in $p=4$ equal to $l(4)=-12+\frac{17\pi^2}{16}+\frac{\pi^4}{256}<0.$ Hence $g$ is decreasing for $p\geq 4$ and we are reduced to prove $g\big(\frac{\pi^2}{16p}\big)\geq \log A_p,$
i.e.
$$\frac{\big(1+\frac{\pi^2}{32p}\big)^{p-1}\big(1+\frac{\pi^2}{96p}
\big)}{\big(1+\frac{\pi^2p}{32(p-1)^2}\big)^{p-2}}\geq \frac{p}{p-1}\cos^p\frac{\pi}{2p}.$$
This is equivalent to
$$\frac{\big(1+\frac{\pi^2}{32p}\big)^{1-\frac{1}{p}}\big(1+\frac{\pi^2}{96p}\big)^{\frac{1}{p}}\big(1-\frac{1}{p}\big)^{\frac{1}{p}}}{\big(1+\frac{\pi^2p}{32(p-1)^2}\big)^{1-\frac{2}{p}}}\geq \cos\frac{\pi}{2p},$$
which, rewriting the left-hand side as
$e^{(1-\frac{1}{p})\log(1+\frac{\pi^2}{32p})+\frac{1}{p}\log(1+\frac{\pi^2}{96p})+\frac{1}{p}\log(1-\frac{1}{p})-(1-\frac{2}{p})\log(1+\frac{\pi^2p}{32(p-1)^2})}$
and using the estimate $e^x\geq 1+x,$ follows from the inequality
\begin{align*}
K(x):&=(1-x)\log\bigg(1+\frac{\pi^2}{32}x\bigg)+x\log\bigg(1+\frac{\pi^2}{96}x\bigg)+x\log\big(1-x\big)\\
&+(2x-1)\log\bigg(1+\frac{\pi^2x}{32(1-x)^2}\bigg)+1-\cos\frac{\pi x}{2}\geq 0,
\end{align*}
for $0\leq x=\frac{1}{p}\leq \frac{1}{4}.$
We use the inequalities $1-\cos\frac{\pi t}{2}\geq \frac{\pi^2t^2}{8}-\frac{\pi^4t^4}{384}$
and $\log(1+t)\geq t-\frac{t^2}{2},$
getting
\begin{align*}
K(x)&\geq L(x)=(1-x)\bigg(\frac{\pi^2}{32} x-\frac{\pi^4x^2}{2048}\bigg)+x\bigg(\frac{\pi^2}{96} x-\frac{\pi^4x^2}{18432}
\bigg)+x\log\big(1-x\big)\\
&+(2x-1)\log\bigg(1+\frac{\pi^2x}{32(1-x)^2}\bigg)+\frac{\pi^2x^2}{8}-\frac{\pi^4x^4}{384}.
\end{align*}
We have 
\begin{align*}
L''(x)&=\frac{\pi^4}{384}x-\frac{\pi^4}{32}x^2+\frac{\pi^2(640-3\pi^2)}{3072}+\frac{4-3x}{(1-x)^2}\\
&+\frac{4(32x+\pi^2-48)\big(32(1-x)^2+\pi^2x\big)+(2x-1)(128\pi^2-\pi^4)}{(32(1-x)^2+\pi^2x)^2}.
\end{align*}
After some straightforward calculations, we get:
$$(1-x)^2\big(32(1-x)^2+\pi^2x\big)^2L''(x)=\sum_{k=0}^{8}a_kx^k,$$
where $a_0>50, a_1>-800, a_2>3900, a_3>500, a_4>-23800, a_5>44800, a_6>-38800, \\a_7>17000, a_8>-3200.$
i.e.
\begin{align*}
&(1-x)^2\big(32(1-x)^2+\pi^2x\big)^2L''(x)\\
&>100\bigg(\frac{1}{2}-8x+39x^2+5x^3-238x^4+448x^5-388x^6+170x^7-32x^8\bigg)\\
&\geq 100\bigg(\frac{1}{2}-8x+39x^2+5x^3-238x^4\bigg)
\end{align*}
It is an easy calculus exercise to check that the last polynomial is positive on $[0,\frac{1}{4}],$
thus $L$ is convex, $L'(x)\geq L'(0)=0$ and $L(x)\geq L(0)=0.$ This concludes the proof of the monotonicity of $f.$\\
 
 Finally, we need to prove that $f(\frac{\pi}{4\sqrt{p}})\geq 0,$ i.e. 
 $$\bigg(1+\frac{\pi^2p}{32(p-1)^2}\bigg)^{p-1}+\tan\frac{\pi}{2p}\geq \frac{\cosh^p\frac{\pi}{4\sqrt{p}}}{\cos^p\frac{\pi}{2p}}.$$
 Since $\cosh\frac{\pi}{4\sqrt{p}}\leq e^{\frac{\pi^2}{32p}}$(an easy comparison of Taylor coefficients), it is enough to establish:
 $$\bigg(1+\frac{\pi^2p}{32(p-1)^2}\bigg)^{p-1}\geq e^{\frac{\pi^2}{32}}\bigg(1+\frac{1}{4p}+\frac{1}{4p^2}\bigg)$$
 and 
 $$1+\frac{1}{4p}+\frac{1}{4p^2}+e^{-\frac{\pi^2}{32}}\tan\frac{\pi}{2p}\geq \frac{1}{\cos^p\frac{\pi}{2p}}.$$

By taking logarithms to both sides of the first inequality and substitution $x=\frac{1}{p}\in[0,\frac{1}{4}]$, it is equivalent to:
$$h(x):=(1-x)\log\bigg(1+\frac{\pi^2x}{32(1-x)^2}\bigg)-\frac{\pi^2}{32}x-x\log\bigg(1+\frac{x}{4}+\frac{x^2}{4}\bigg)\geq 0.$$
We have:
$$(1-x)(x^2+x+4)^2(\pi^2x+32(1-x)^2)^2h''(x)=\sum_{k=0}^{8}b_kx^k,$$
where  $b_0>355, b_1>-5700, b_2>43800, b_3>-133700, b_4>152200, b_5>-87500, b_6>25000, \\b_7>-4900, b_8=2048$
and
\begin{align*}
&(1-x)(x^2+x+4)^2(\pi^2x+32(1-x)^2)^2h''(x)\\
&\geq 100\bigg(\frac{7}{2}-57x+438x^2-1337x^3+1522x^4-875x^5+250x^6-49x^7+20x^8\bigg)\\
&\geq 100\bigg(\frac{7}{2}-57x+438x^2-1337x^3+1300x^4\bigg),
\end{align*}
which is positive on $(0,\frac{1}{4}).$ Thus $h''(x)>0,$ implying $h'(x)\geq h'(0)=0$ and $h(x)\geq h(0)=0$
and the desired estimate follows. \\
  
For the second inequality, we use $\tan x\geq x+\frac{x^3}{3},$ therefore, by
$e^{-\frac{\pi^2}{32}}\frac{\pi}{2}>\frac{23}{20}$ and $e^{-\frac{\pi^2}{32}}\frac{\pi^3}{24}>\frac{9}{10},$ it is sufficient to show
$$1+\frac{7}{5p}+\frac{1}{4p^2}+\frac{9}{10p^3}\geq \frac{1}{\cos^p\frac{\pi}{2p}}.$$ 
This is equivalent to
$$k(x):=x\log\bigg(1+\frac{7x}{5}+\frac{x^2}{4}+\frac{9x^3}{10}\bigg)+\log\big(\cos\frac{\pi x}{2}\big)\geq 0.$$
We find
$$k''(x)=\frac{2Q(x)}{P(x)^2}-\frac{\pi^2}{4\cos^2\frac{\pi x}{2}},$$
where
$$P(x)=18x^3+5x^2+28x+20 \quad \text{and}\quad Q(x)=486x^5+270x^4+2041x^3+2440x^2+692x+500.$$ 
Since the sign of $k''$ is oposite to that of 
$$\tilde{k}(x):=\frac{P^2}{2Q}-\frac{4}{\pi^2}\cos^2\frac{\pi x}{2},$$
we will inspect the sign of $\tilde{k}.$
From
$$\tilde{k}'(x)=(2P'Q-PQ')\frac{P}{2Q^2}+\frac{2}{\pi}\sin \pi x,$$
we see that $\tilde{k}'(x)>0$ for those $x$ such that $2P'Q-Q'P\geq 0.$ But, it decreases on $(0,\frac{1}{4})$, hence, it has exactly one zero $x_0$ and $\tilde{k}'$ is positive on $(0,x_0).$  Since
\begin{align*}
2P'Q-PQ'&=8748x^7+7290x^6+69390x^5+122165x^4+7932x^3-58080x^2-68224x+14160\\
&\geq -58080x^2-68224x+14160 \geq -\frac{58080}{36}-\frac{68224}{6}+14160>0,
\end{align*}
we infer that $x_0>\frac{1}{6}$. On the other hand, for $x_0\leq x\leq\frac{1}{4},$ where $2P'Q-Q'P<0,$
we get:
\begin{align*}
\tilde{k}'(x)&=(2P'Q-PQ')\frac{P}{2Q^2}+\frac{2}{\pi}\sin \pi x \geq \frac{P(\frac{1}{4})}{2Q^2(\frac{1}{6})}(2P'Q-PQ')+2x-\frac{\pi^2x^3}{3}\\
&\geq 3\cdot10^{-5}(7932x^3-58080x^2-68224x+14160)+2x-\frac{\pi^2x^3}{6},
\end{align*}
which is easily seen to be positive.
Therefore, $\tilde{k}$ increases, and since $\tilde{k}(0)<0 $ and $\tilde{k}\big(\frac{1}{4}\big)>0,$
$k''$ is positive till some point, then negative, $k'$ increases and then decreases, $k'(0)=0,$ $k'\big(\frac{1}{4}\big)>0,$ giving that $k$ increases on $[0,\frac{1}{4}]$ and $k(x)\geq k(0)=0.$ This concludes the proof.\\

\end{proof}

\section{Ethics declarations}

\textbf{Conflict of interest.} The author declares that he has not conflict of interest.\\

\textbf{Data statement.} Data sharing not applicable to this article as no datasets were generated or analysed during the current study.


\end{document}